\documentclass{amsart}

\usepackage{amsthm,amsmath,amssymb,amscd,amsfonts,latexsym}

\theoremstyle{plain}
\newtheorem{theorem}{Theorem}[section]

\newtheorem{corollary}[theorem]{Corollary}
\newtheorem{def-thm}[theorem]{Definition-Theorem}
\newtheorem{lemma}[theorem]{Lemma}

\theoremstyle{definition}

\newtheorem{remark}[theorem]{Remark}

\newcommand{\PP}{\mathbb{P}}

\newcommand{\RR}{\mathbb{R}}

\newcommand{\CC}{\mathbb{C}}
\newcommand{\FF}{\mathbb{F}}

\newcommand{\OO}{{\mathcal O}}

\DeclareMathOperator{\vol}{vol}

\begin{document}

\title[Optimal pinching for the holomorphic sectional curvature]{Optimal pinching for the holomorphic sectional curvature of Hitchin's metrics on Hirzebruch surfaces}

\begin{abstract}
The main result of this note is that, for each $n\in \{1,2,3,\ldots\}$, there exists a Hodge metric on the $n$-th Hirzebruch surface whose positive holomorphic sectional curvature is $\frac{1}{(1+2n)^2}$-pinched. The type of metric under consideration was first studied by Hitchin in this context. In order to address the case $n=0$, we prove a general result on the pinching of the holomorphic sectional curvature of the product metric on the product of two Hermitian manifolds $M$ and $N$ of positive holomorphic sectional curvature.
\end{abstract}

\author{Angelynn Alvarez, Ananya Chaturvedi, Gordon Heier}

\address{Department of Mathematics\\University of Houston\\4800 Calhoun Road, Houston, TX 77204\\USA}
\email{aalvarez@math.uh.edu}
\email{ananya@math.uh.edu}
\email{heier@math.uh.edu}

\thanks{The third named author is partially supported by the National Security Agency under Grant Number H98230-12-1-0235. The United States Government is authorized to reproduce and distribute reprints notwithstanding any copyright notation herein.}

\maketitle

\section{Introduction}
It is a well-known fact that the Fubini-Study metric on a complex projective space of arbitrary dimension has constant holomorphic sectional curvature equal to $4$. However, in general, few examples are known of compact complex manifolds which carry a Hermitian metric of positive holomorphic sectional curvature, let alone a Hermitian metric with positively pinched holomorphic sectional curvature. A notable exception form the irreducible Hermitian symmetric spaces of compact type, whose pinching constants for the holomorphic sectional curvature are listed in \cite[Table I]{Chen_table} (see also the references in that paper). In particular, the geometry and curvature of fibrations and even fiber bundles are poorly understood in this respect.\par
In this note, we are primarily interested in the Hirzebruch surfaces $\FF_n=\PP(\OO_{\PP^1}(n) \oplus  \OO_{\PP^1})$, $n\in \{0,1,2,\ldots\}$. It was proven by Hitchin in \cite{Hitchin} that they do carry a natural metric of positive holomorphic sectional curvature, but his proof does not yield any pinching constants. Even this nonquantitative positivity result may be considered to be somewhat surprising, as the $\FF_n$ do not carry metrics of positive Ricci curvature, except when $n=0$, or $n=1$. Our main result is the following pinching theorem for the metrics on $\FF_n$ considered in \cite{Hitchin}, whose definition is recalled in Section \ref{sec_def}.
\begin{theorem}\label{mthm}
Let $\FF_n$, $n\in \{1,2,3,\ldots\}$, be the $n$-th Hirzebruch surface. Then there exists a Hodge metric on $\FF_n$ whose positive holomorphic sectional curvature is $\frac{1}{(1+2n)^2}$-pinched.
\end{theorem}
We also prove that the numerical values of the pinching constants are optimal in the families of metrics studied by Hitchin. This does however leave open the question if there are other types of metrics on Hirzebruch surfaces with better pinching constants. Recall that an upper bound on the possible value of such pinching constants was given in the paper \cite{bishop_goldberg}, where it was proven that a complete K\"ahler manifold whose positive holomorphic sectional curvature is $c$-pinched with $c>\frac 4 5$ is homotopic to a complex projective space.\par
The proofs work by way of explicit computations, using in particular the method of Lagrange multipliers. Our results can likely be generalized to projectivized vector bundles of higher rank over higher-dimensional bases, but the computations will surely become much more involved, and we will leave this for a later occasion.\par
Since we could not find a reference for it, we also include the following pinching theorem for products $M \times N$ of Hermitian manifolds endowed with the product metric. If $M=N=\PP^1$, then this theorem addresses the case of the $0$-th Hirzebruch surface $\PP^{1}\times \PP^{1}$, which was not handled in Theorem \ref{mthm}. In this case, $c_M=c_N=c_{\PP^1}=1$, $k=4$, and $\frac{c_Mc_N}{c_M+c_N}=\frac 1 2$.
\begin{theorem}\label{prod_thm}
Let $M$ and $N$ be Hermitian manifolds whose positive holomorphic sectional curvatures are $c_M$- and $c_N$-pinched respectively and satisfy
$$kc_M\leq K_M \leq k\ \ \text{and}\ \ kc_N\leq K_N \leq k$$
for a constant $k>0$. Then the holomorphic sectional curvature $K$ of the product metric on $M\times N$ satisfies
$$k\frac{c_Mc_N}{c_M+c_N}\leq K \leq k$$
and is $\frac{c_Mc_N}{c_M+c_N}$-pinched.
\end{theorem}
Recall that the Hopf Conjecture states that the product of two real two-spheres does not admit a Riemannian metric of positive sectional curvature, so even the case of products as in Theorem \ref{prod_thm} is not trivial with respect to sectional curvatures. \par
This paper is organized as follows. In Section \ref{sec_def}, we will recall fundamental definitions and establish our basic setup. In Section \ref{proof_mthm}, we will prove Theorem \ref{mthm} and also derive a corollary giving lower and upper bounds for the scalar curvature of the metrics under investigation. In Section 4, we will give an interpretation of the results of our computations in terms of the geometry of Hirzebruch surfaces. In Section \ref{proof_prod_thm}, we will prove Theorem \ref{prod_thm}.
 \section{Basic definitions and description of the family of metrics under consideration}\label{sec_def}
Let $M$ be an $m$-dimensional manifold with local coordinates $z_1,\ldots,z_m$. Let
\begin{equation*}
g=\sum_{i,j=1}^m g_{i\bar j} dz_i\otimes d\bar{z}_j
\end{equation*}
be a Hermitian metric on $M$. Under the usual abuse of terminology, we will alternatively refer to the associated (1,1)-form $\omega=\frac{\sqrt{-1}}{2}\sum_{i,j=1}^m g_{i\bar{j}} dz_i\wedge d\bar{z}_{j}$ as the metric on $M$. The metric is called {\it K\"ahler} if $\omega$ is $d$-closed. It is called {\it Hodge} if it is K\"ahler and the cohomology class of $\omega$ is rational.\par
The components $R_{i\bar j k \bar l}$ of the curvature tensor $R$ associated with the metric connection are locally given by the formula
\begin{equation}\label{curv_formula}
R_{i\bar j k \bar l}=-\frac{\partial^2 g_{i\bar j}}{\partial z_k\partial \bar z _l}+\sum_{p,q=1}^m g^{p\bar q}\frac{\partial g_{i\bar p}}{\partial z_k}\frac{\partial g_{q\bar j}}{\partial \bar z_l}.
\end{equation}\par
If $\xi=\sum_{i=1}^m\xi_i \frac{\partial }{\partial z_i}$ is a non-zero complex tangent vector at $p\in M$, then the {\it holomorphic sectional curvature} $K(\xi)$ is given by
\begin{equation}\label{hol_sect_curv_def}
K(\xi)=\left( 2 \sum_{i,j,k,l=1}^m R_{i\bar j k \bar l}(p)\xi_i\bar\xi_j\xi_k\bar \xi_l\right) / \left(\sum_{i,j,k,l=1}^m g_{i\bar j}g_{k\bar l} \xi_i\bar\xi_j\xi_k\bar \xi_l\right).
\end{equation}
Note that the holomorphic sectional curvature of $\xi$ is clearly invariant under multiplication of $\xi$ with a real non-zero scalar, and it thus suffices to consider unit vectors, for which the value of the denominator is $1$. For a constant $c\in (0,1]$, we say that the (positive) holomorphic sectional curvature is $c$-{\it pinched} if
$$(1\geq)\ \frac{\inf_\xi K(\xi)}{\sup_\xi K(\xi)} = c,$$
where the infimum and supremum are taken over all non-zero (or unit) tangent vectors across the entire manifold. In the case of a compact manifold, the infimum and supremum become a minimum and maximum, respectively, due to compactness.\par
Moreover, it is a basic fact that the holomorphic sectional curvature of a K\"ahler metric completely determines the curvature tensor $R_{i\bar j k \bar l}$ (\cite[Proposition 7.1, p.\ 166]{kobayashi_nomizu_ii}). However, as we remarked in the introduction, positivity or negativity properties of the holomorphic sectional curvature of a K\"ahler metric do not necessarily transfer to the {\it Ricci curvature} $R_{i\bar j}$, which is defined as the following trace of the curvature tensor:
$$R_{i\bar j}= \sum_{k,l=1}^m g^{k\bar l}R_{i\bar j k \bar l}.$$
Nevertheless, there is a beautiful integral formula due to Berger (see Lemma \ref{berger_lemma}) which expresses the {\it scalar curvature} $\tau$ of a K\"ahler metric as an integral of the holomorphic sectional curvature, while the standard definition is as the trace of the Ricci curvature:
$$\tau = \sum_{i,j=1}^m g^{i\bar j} R_{i\bar j}= \sum_{i,j,k,l=1}^m g^{i\bar j} g^{k\bar l}R_{i\bar j k \bar l}.$$\par
Following Hitchin's idea from \cite{Hitchin}, we recall that on the $n$-th Hirzebruch surface $\FF_{n}$, there are natural Hermitian metrics defined as follows. Note that these metrics are clearly K\"ahler and, when the value of the parameter $s$ is rational, even Hodge. \par
If $z_1$ is an inhomogeneous coordinate on an open subset of the base space $\PP^1$, then a point
$$w \in\OO_{\PP^1}(n)\oplus \OO_{\PP^1}$$
can be represented by coordinates $w_1,w_{2}$ in the fiber direction as
$$w=(z_1,w_1(dz_1)^{-n/2},w_{2}),$$
where $(dz_1)^{-1}$ is to be understood as a section of $T\PP^1 =  \OO_{\PP^1}(2)$.
After the projectivization, each fiber carries the inhomogeneous coordinate $z_2= w_2/w_1$. For a positive real number $s$, the metric
$$\omega_s = \frac{\sqrt{-1}}{2}\partial \bar \partial (\log (1+z_1\bar z_1) + s\log ((1+z_1\bar z_1)^n+z_2\bar z_2))$$
is globally well-defined on $\FF_{n}$. It is this metric for which we compute the holomorphic sectional curvature pinching. We also find the choice of $s$ with the optimal value of the pinching constant in the family of metrics parametrized by $s$.
\begin{remark}
In \cite{Hitchin}, the curvature tensor is expressed in terms of a local unitary frame field. In this note, we prefer to work in terms of the frame field $\frac{\partial}{\partial z_1}, \frac{\partial}{\partial z_2}$ with respect to the coordinates discussed above, as it seems to lend itself better to our method. \par
\end{remark}
\section{Proof of Theorem \ref{mthm}}\label{proof_mthm}
\subsection{The case $n\ge 2$}
As observed in \cite{Hitchin}, the fact that $SU(2)$ acts transitively on $\PP^1$ as isometries of the Fubini-Study metric and that this action lifts to $\OO_{\PP^1}(n)\oplus \OO_{\PP^1}$, implies that we can restrict ourselves to computing the curvature along one fiber, say the one given by $z_1=0$.
The metric tensor associated to $\omega_s$ along this fiber is
$$(g_{i \bar j}) =  \left( \begin{array}{cc} \frac{1+z_2\bar{z}_2+sn}{1+z_2\bar{z}_2} & 0\\ 0 & \frac{s}{(1+z_2\bar{z}_2)^2} \end{array} \right).$$
From this, we see that an orthonormal basis for $T_{(0, z_2)}\mathbb{F}_n$ is given by the two vectors
$$\sqrt{\frac{1+z_2\bar{z}_2}{1+z_2\bar{z}_2+ns}} \cdot \frac{\partial}{\partial z_1}\ \text{ and }\ \frac{1+z_2\bar{z}_2}{\sqrt{s}} \cdot \frac{\partial}{\partial z_2}.$$
Therefore, an arbitrary unit tangent vector $\xi \in T_{(0, z_2)} \mathbb{F}_n$ can be written as
$$\xi = c_1 \sqrt{\frac{1+z_2\bar{z}_2}{1+z_2\bar{z}_2+ns}} \cdot \frac{\partial}{\partial z_1} + c_2 \frac{1+z_2\bar{z}_2}{\sqrt{s}} \cdot \frac{\partial}{\partial z_2},$$
where $c_1, c_2 \in \CC$ are such that $|c_1|^2+|c_2|^2=1$.
Let $\xi_1:=  c_1 \sqrt{\frac{1+z_2\bar{z}_2}{1+z_2\bar{z}_2+ns}}$ and $\xi_2:=c_2 \frac{1+z_2\bar{z}_2}{\sqrt{s}}$.
Based on the formula \eqref{curv_formula} in Section \ref{sec_def}, the components of the curvature tensor are
\begin{align*}
R_{1 \bar 1 1 \bar 1} & = \frac{2(-n^2sz_2\bar{z}_2+(1+z_2\bar{z}_2)^2+n(s+sz_2\bar{z}_2))}{(1+z_2\bar{z_2})^2},\\
R_{1 \bar 1 2 \bar 2} & =\frac{ns(1+ns-z_2^2\bar{z}_2^2)}{(1+z_2\bar{z}_2)^3(1+ns+z_2\bar{z}_2)},\\
R_{2 \bar 2 2 \bar 2} & =\frac{2s}{(1+z_2\bar{z}_2)^4},
\end{align*}
while the other terms (except those obtained from symmetry) are zero. Substituting the components and values of $\xi_1$ and $\xi_2$ into the definition \eqref{hol_sect_curv_def} of holomorphic sectional curvature in the direction of $\xi$ gives us
\begin{align*}  K(\xi) & =  2 \sum_{i,j,k,l=1}^2 R_{i \bar j k \bar l} \xi_i \bar{\xi}_j \xi_k \bar{\xi}_l\\
& =  2R_{1 \bar 1 1 \bar 1} \xi_1 \bar{\xi}_1 \xi_1 \bar{\xi}_1 + 8R_{1 \bar 1 2 \bar 2} \xi_1 \bar{\xi}_1 \xi_2 \bar{\xi}_2 + 2R_{2 \bar 2 2 \bar 2} \xi_2 \bar{\xi}_2 \xi_2 \bar{\xi}_2\\
& =  \frac{4((1+z_2\bar{z}_2)^2+ns(1+z_2\bar{z}_2 - nz_2\bar{z}_2))}{(1+z_2\bar{z}_2+ns)^2}|c_1|^4 \\
&\quad + \frac{8n(1+ns-z_2^2\bar{z}_2^2)}{(1+z_2\bar{z}_2+ns)^2}|c_1|^2|c_2|^2+ \frac{4}{s}|c_2|^2.
\end{align*}
Since the above expression only depends on the modulus squared of $z_2$, we let $r:= z_2\bar{z}_2$. Also, we let $a:= |c_1|^2$ and $b:=|c_2|^2$, satisfying $a+b=1$ and $a,b\in [0,1]$. Hence, for fixed values of $r$ and $s$, the holomorphic sectional curvature takes the form of a degree two homogeneous polynomial in $a$ and $b$ with real coefficients:
\begin{equation}\label{hol_sect_curv_simplified}
K_{r,s}(a,b)= \frac{4((1+r)^2+ns(1+r - nr))}{(1+r+ns)^2}a^2 + \frac{8n(1+ns-r^2)}{(1+r+ns)^2}ab+ \frac{4}{s}b^2.
\end{equation}
We write $\alpha:=   \frac{4((1+r)^2+ns(1+r - nr))}{(1+r+ns)^2},\ \beta:=\frac{8n(1+ns-r^2)}{(1+r+ns)^2}$, and $\gamma:=  \frac{4}{s}$ for the coefficients.\par
In order to find the pinching constant for the metric $\omega_s$, we need to minimize and maximize
\begin{equation*}
K_{r,s}(a,b)= \alpha a^2+ \beta ab + \gamma b^2
\end{equation*}
for fixed $s$, subject to the constraint $a+b-1=0$. To do so, we first also fix $r$ and set up the Lagrange Multiplier equations:
\begin{align*}
\frac{\partial}{\partial a} K_{r,s}(a,b) = \lambda,\ \frac{\partial}{\partial b} K_{r,s}(a,b) = \lambda,\ a+b-1 = 0.
\end{align*}
Solving this system of equations for $a,b$ yields a unique stationary solution
\begin{align*}
a_0 &=\frac{2\gamma-\beta}{2(\gamma-\beta+\alpha)} = \frac{(1+r)(1+ns)}{1+s-(-1+n)ns^2+r(1+s+2ns)},\\
b_0 &= \frac{2\alpha-\beta}{2(\gamma-\beta+\alpha)} = \frac{s(1-n+r+nr+ns-n^2s)}{1+s-(-1+n)ns^2+r(1+s+2ns)}.
\end{align*}
Substituting these values into equation \eqref{hol_sect_curv_simplified} gives us
\begin{align*}
& K_{r,s}(a_0,b_0) \\ = &\ 4 \cdot \frac{3r^2(1+ns)+3r(1+ns)^2-r^3(-1+n^2s)-(1+ns)^2(-1-ns+n^2s)}{(1+r+ns)^2(1+s-(-1+n)ns^2+r(1+s+2ns))}.
\end{align*}
We shall now find lower and upper bounds for the holomorphic sectional curvature in the following three cases:
\begin{enumerate}
  \item \emph{For $a = a_0$ and $b = b_0$:} For a fixed value of $s$, define $f_s : [0,\infty) \to \RR$ as
        \[f_s(r):=K_{r,s}(a_0,b_0).\]
        A computation yields that $f_s'(r)=0$ if and only if $r=-1 \notin (0,\infty)$ (which we may disregard) or
        \[r=r_0:=\frac{(n-1)(1+ns)}{1+n},\]
        which is in $(0,\infty)$ under the assumption $n  \geq 2$. Note
        \[f_s(r_0) = \frac{4-s(n-1)^2}{1+ns}.\]
        At the endpoints of the interval $[0,\infty)$, we see that
        \[f_s(0)=\frac{4(1+ns-n^2s)}{1+s-(n-1)ns^2}, \quad \text{and} \quad \lim_{r\to\infty} f_s(r)=\frac{4-4n^2s}{1+s+2ns}.\]
        The latter expression makes it clear that we need to choose $s < \frac 1 {n^2}$ in order to obtain positive holomorphic sectional curvature.   Furthermore, for $s < \frac 1 {n^2}$,
        \[\frac{4(1+ns-n^2s)}{1+s-(n-1)ns^2}-\frac{4-4n^2s}{1+s+2ns}=\frac{4s(3n-s(2n^3-3n^2)-s^2(n^4-n^3))}{(1+s+2ns)(1+s(1-s(n^2-n)))} >0,\]
 and        \[\frac{4-s(n-1)^2}{1+ns} - \frac{4(1+ns-n^2s)}{1+s-(n-1)ns^2} = -\frac{s(n-1)^2(3+s(n-1))}{(-1+s(n-1))(1+ns)}>0.\]
        Thus,
        \[\frac{4-s(n-1)^2}{1+ns} > \frac{4(1+ns-n^2s)}{1+s-(n-1)ns^2} > \frac{4-4n^2s}{1+s+2ns}.\]
  \item \emph{For $a=0$ and $b=1$:} The curvature value is $K_{r,s}(0,1)=\frac{4}{s}$, which is independent of $r$.
  \item \emph{For $a=1$ and $b=0$:} The curvature value is
        \[h_s(r):=K_{r,s}(1,0)=\frac{4((1+r)^2+ns(1+r-nr))}{(1+r+ns)^2}.\]
        In the interval $(0, \infty)$, we have that $h_s'(r)=0$ if and only if
        \[r=r_0=\frac{(n-1)(1+ns)}{1+n} \quad(\in (0, \infty)\ \text{when}\ n\geq 2),\]
        with
        \[h_s(r_0) = \frac{4-s(n-1)^2}{1+ns}.\]
        Note that this is the same $r_0$ as above, although we see no clear geometric reason for this coincidence. At the endpoints, we have
        \[h_s(0)=\frac{4}{1+ns}, \quad \text{and} \lim_{r\to\infty} h_s(r)=4.\]
        Clearly, we have
        \[4 > \frac{4}{1+ns} > \frac{4-s(n-1)^2}{1+ns}.\]
\end{enumerate}
Combining the three cases above, we have for $n \geq 2$:
\[\frac{4}{s} > 4 > \frac{4}{1+ns} > \frac{4-s(n-1)^2}{1+ns} > \frac{4(1+ns-n^2s)}{1+s-(n-1)ns^2} > \frac{4-4n^2s}{1+s+2ns}.\]
Hence, the smallest and largest values attained by the holomorphic sectional curvature are
\[\lim_{r\to\infty} f_s(r)=\frac{4-4n^2s}{1+s+2ns} \quad \text{and}\quad  \frac{4}{s},\]
respectively.\par
Finally, in order to find the value of $s$ with the best pinching constant, we define a function
$$p : (0,\frac 1 {n^2}) \to (0,1),\ p(s) := \frac{\min_\xi K_s(\xi)}{\max_\xi K_s(\xi)} = \frac{\frac{4-4n^2s}{1+s+2ns}}{\frac{4}{s}}=\frac{s(1-n^2s)}{1+s+2ns},$$
where the minimum and maximum are taken over all non-zero (or unit) tangent vectors across the entire manifold and the index $s$ indicates that the holomorphic sectional curvature is computed with respect to the metric with the parameter value $s$. This is the function which we want to maximize. We see that $p'(s)=0$ if and only if $s=-\frac{1}{n} \notin (0,\frac 1 {n^2})$ or $s=\frac{1}{2n^2+n}\in (0,\frac 1 {n^2})$. Elementary calculus tells us that $p$ has a global maximum at $\frac{1}{2n^2+n}$. Hence, with $s=\frac{1}{2n^2+n}$, we get the optimal pinching of 
$$p\left(\frac{1}{2n^2+n}\right)=\frac{1}{(1+2n)^2}.$$

\subsection{The case $n=1$}
In the case when $n=1$, the functions $f_s$ and $h_s$ have their stationary points at the boundary point $r=0$. However, our reasoning still goes through almost verbatim and yields the expected pinching constant $\frac 1 9$ for $s= \frac 1 3$.

\subsection{A remark on scalar curvature}
The following formula due to \cite[Lemme 7.4]{berger} expresses the scalar curvature of a K\"ahler manifold as an integral of the holomorphic sectional curvature.
\begin{lemma}\label{berger_lemma}
Let $M$ be an $m$-dimensional K\"ahler manifold. Then the scalar curvature $\tau$ satisfies at every point $P\in M$:
$$\tau(P)=\frac{m(m+1)}{4\vol(S_P^{2m-1})}\int_{\xi\in S_P^{2m-1}}K(\xi) d\xi,$$
where $S_P^{2m-1}$ denotes the unit sphere inside the tangent space $T_P M$ with respect to the metric, and $d\xi$ is the measure on $S_P^{2m-1}$ induced by the metric.
\end{lemma}
This lemma yields the following corollary.
\begin{corollary}
Let $\tau_s$ denote the scalar curvature of  $\FF_n$, $n\in \{1,2,3,\ldots\}$, pertaining to the metric $\omega_s$. Then
$$\frac 3 2 \min_\xi K_s(\xi)=\frac 3 2 \cdot \frac{4-4n^2s}{1+s+2ns}  \leq \tau_s \leq \frac 3 2\cdot \frac 4 s =  \frac 3 2  \max_\xi K_s(\xi).$$
In particular, for our optimal choice of $s=\frac{1}{2n^2+n}$, we have
$$\frac{6n(n+1)}{2n^2+3n+1} \leq \tau \leq 12n^2+6n.$$
\end{corollary}
\begin{proof}
The proof is immediate from Lemma \ref{berger_lemma} and the bounds for the holomorphic sectional curvature: Replace the integrand $K_s(\xi)$ by the minimum and maximum, respectively, which we computed, move the constant in front of the integral, cancel $\vol(S_P^{2m-1})$, and let $m=2$.
\end{proof}
Finally, since the scalar curvature is additive in products equipped with the product metric, and since the scalar curvature of $\PP^1$ with the Fubini-Study metric is constant and equal to $2$, it is immediately clear that the scalar curvature of $\FF_0=\PP^{1}\times \PP^{1}$ is constant and equal to $4$.
\section{Geometric interpretation of our computations}
The Hirzebruch surfaces have a beautiful geometric structure, which is very nicely explained in \cite[pp.\ 517--520]{GH}. In particular, on the $n$-th Hirzebruch surface, there is a unique non-singular rational curve $E$ ``at infinity" which has self-intersection number $-n$. In terms of our coordinates $z_1,z_2$, the curve $E$ is given by $z_2 = \infty$.  The fact that
$$\min_\xi K_s(\xi)= \lim_{r\to\infty} f_s(r)$$
means that the smallest value of the holomorphic sectional curvature for each $\omega_s$ is attained at a tangent vector attached to a point of $E$. Note that because of the transitivity of the $SU(2)$ action, this is then true for all points of $E$.  Since the largest value $\frac 4 s$ is attained inside every tangent space of $\FF_n$, every point $P\in E$ has the property that the tangent space to $\FF_n$ at $P$ contains a vector giving the lowest possible holomorphic sectional curvature and a vector giving the highest possible holomorphic sectional curvature. In other words, for Hirzebruch surfaces, the notion of the ``pinching constant" and the ``pointwise pinching constant" are one and the same. \par
We can still say more about the vectors yielding the extreme values. If we consider $a_0$ and $b_0$ as functions of $r$ and set $s=\frac{1}{2n^2+n}$, then
$$\lim_{r\to \infty} a_0 = \frac{1+ns}{1+s+2ns}=\frac{2n}{2n+1},\ \text{and } \lim_{r\to \infty} b_0 = \frac{s(1+n)}{1+s+2ns}=\frac{1+n}{1+3n+2n^2}.$$
For large values of $n$, the first value is a little less than $1$, and the second value is a little larger than $0$. This means that the direction of the tangent vector giving the smallest value of the holomorphic sectional curvature is close, but not equal, to the direction of the tangent space of $E$, which we think of as the ``horizontal" direction. Moreover, the direction of the tangent vector giving the largest value of the holomorphic sectional curvature is exactly ``vertical" and thus almost, but not exactly, perpendicular to the direction giving the smallest value.

\section{Proof of Theorem \ref{prod_thm}}\label{proof_prod_thm}
The proof of Theorem \ref{prod_thm} consists of computing the holomorphic sectional curvature of the product metric on the product $M^m\times N^n$, $m,n\in\{1,2,3,\ldots\}$, of two Hermitian manifolds with local coordinates $z_1,\ldots,z_m$ and $z_{m+1},\ldots,z_{m+n}$ around points $P \in M$ and $Q \in N$, respectively. Let
$g=\sum_{i,j=1}^{m} g_{i \bar j} dz_i \otimes d\bar z_j$, and $h=\sum_{i,j=m+1}^{m+n} h_{i \bar j} dz_i \otimes d\bar z_j$ be Hermitian metrics on $M$ and $N$, respectively, with positive holomorphic sectional curvature. Then
$$\sum_{i,j=1}^{m} g_{i \bar j} dz_i \otimes d\bar z_j + \sum_{i,j=m+1}^{m+n} h_{i \bar j} dz_i \otimes d\bar z_j$$
gives the product metric in a neighborhood of $(P, Q) \in M \times N$.
Since the $g_{i \bar j}$ are functions of only $z_1,\ldots,z_m$ and the $h_{i \bar j}$ are functions of only $z_{m+1},\ldots,z_{m+n}$, we obtain
\begin{equation*}
R_{i \bar j k \bar l} = \begin{cases}
                       -\frac{\partial^2 g_{i \bar j}}{\partial z_k\partial \bar z _l}+\sum_{p,q=1}^m g^{p \bar q}\frac{\partial g_{i \bar p}}{\partial z_k}\frac{\partial g_{q \bar j}}{\partial \bar z_l}, & 1 \leq i,j,k,l \leq m \\
                       \\ -\frac{\partial^2 h_{i \bar j}}{\partial z_k\partial \bar z _l}+\sum_{p,q=m+1}^{m+n} h^{p \bar q}\frac{\partial h_{i \bar p}}{\partial z_k}\frac{\partial h_{q \bar j}}{\partial \bar z_l}, & m+1 \leq i,j,k,l \leq m+n \\
                      \\ 0, & \text{otherwise}.
                     \end{cases}
\end{equation*}\par
Let $\xi=\sum_{i=1}^{m+n} \xi_{i} \frac{\partial}{\partial z_i}$ be a  unit tangent vector in $T_{(P,Q)}(M\times N)$. Then the holomorphic sectional curvature on $M \times N$ along $\xi$ is
\begin{align*}
  K(\xi) = & \; 2\sum_{i,j,k,l=1}^m \bigg(-\frac{\partial^2 g_{i \bar j}}{\partial z_k\partial \bar z _l}+\sum_{p,q=1}^m g^{p \bar q}\frac{\partial g_{i \bar p}}{\partial z_k}\frac{\partial g_{q \bar j}}{\partial \bar z_l}\bigg)\xi_i\bar{\xi}_j\xi_k\bar{\xi}_l \\
           & + 2\sum_{i,j,k,l=m+1}^{m+n} \bigg(-\frac{\partial^2 h_{i \bar j}}{\partial z_k\partial \bar z _l}+\sum_{p,q=m+1}^{m+n} h^{p \bar q}\frac{\partial h_{i \bar p}}{\partial z_k}\frac{\partial h_{q \bar j}}{\partial \bar z_l}\bigg)\xi_i\bar{\xi}_j\xi_k\bar{\xi}_l.
\end{align*}
The two sums on the right hand side above are the numerators of the holomorphic sectional curvatures on $M$ and $N$ with respect to the tangent vectors $(\xi_1,\ldots,\xi_m) \in T_P M$ and $(\xi_{m+1},\ldots,\xi_{m+n}) \in T_Q N$, respectively, both of which are positive. Thus,
$$K(\xi) > 0.$$\par
In order to find the pinching constant, we need to take into consideration the (non-zero) norms of $(\xi_1,\ldots,\xi_m) \in T_P M$ and $(\xi_{m+1},\ldots,\xi_{m+n}) \in T_Q N$ with respect to the respective metrics in the two spaces, as follows:
\begin{align*}
K(\xi) = & \; \sum_{i,k,j,l=1}^{m} 2R_{i \bar j k \bar l} \xi_i \bar \xi_j \xi_k \bar \xi_l + \sum_{i,k,j,l=m+1}^{m+n} 2R_{i \bar j k \bar l} \xi_i \bar \xi_j \xi_k \bar \xi_l \\
         = & \; \frac{\sum_{i,k,j,l=1}^{m} 2R_{i \bar j k \bar l} \xi_i \bar \xi_j \xi_k \bar \xi_l}{\sum_{i,j,k,l=1}^m g_{i \bar j}g_{k \bar l} \xi_i \bar \xi_j \xi_k \bar\xi_l} \cdot \sum_{i,j,k,l=1}^m g_{i \bar j}g_{k \bar l} \xi_i \bar \xi_j \xi_k \bar\xi_l \\
           & + \frac{\sum_{i,k,j,l=m+1}^{m+n} 2R_{i \bar j k \bar l} \xi \bar \xi_j \xi_k \bar \xi_l}{\sum_{i,j,k,l=m+1}^{m+n} h_{i \bar j}h_{k \bar l} \xi_i \bar \xi_j \xi_k \bar\xi_l} \cdot \sum_{i,j,k,l=m+1}^{m+n} h_{i \bar j}h_{k \bar l} \xi_i \bar \xi_j \xi_k \bar\xi_l \\
         = & \; K_M \cdot y^2 + K_N \cdot (1-y)^2,
\end{align*}
where $K_M$ is the holomorphic sectional curvature of $M$ along $(\xi_1,\ldots,\xi_m)$, $K_N$ the holomorphic sectional curvature of $N$ along $(\xi_{m+1},\ldots,\xi_{m+n})$ and $y=\sum_{i,j}^m g_{i \bar j} \xi_i \bar \xi_j$.\\ Since $\xi$ is a unit tangent vector in $T_{(P,Q)}(M \times N)$, i.e., $\sum_{i,j=1}^m g_{i \bar j} \xi_i \bar \xi_j + \sum_{i,j=m+1}^{m+n} h_{i \bar j} \xi_i \bar \xi_j = 1$, we have $$\sum_{i,j=m+1}^{m+n} h_{i \bar j} \xi_i \bar \xi_j = 1 - \sum_{i,j=1}^m g_{i \bar j} \xi_i \bar \xi_j = 1-y.$$
Furthermore, the assumption
$$kc_M \leq K_M\leq k \; \; \; \text{  and  } \; \; \; kc_N \leq K_N \leq k$$
provides the following inequality:
$$F(y) := kc_M y^2+kc_N(1-y)^2 \leq K_My^2+K_N(1-y)^2 \leq ky^2+k(1-y)^2 =: \widetilde{F}(y).$$
Finally, elementary calculus yields $$\min_{0\leq y \leq 1}{F(y)} = k\frac{c_M c_N}{c_M+c_N}$$ and $$\max_{0 \leq y \leq 1} {\widetilde{F}(y)} = k.$$
In particular, $$k\frac{c_M c_N}{c_M+c_N} \leq K(\xi) \leq k,$$
and the pinching constant for the holomorphic sectional curvature on the product space is obtained as $$c_{M \times N} = \frac{\inf_{\xi}K(\xi)}{\sup_{\xi}K(\xi)} = \frac{c_M c_N}{c_M+c_N}.$$

\end{document}